\documentclass[12pt,reqno]{article}

\usepackage[usenames]{color}
\usepackage{amssymb}
\usepackage{amsmath}
\usepackage{amsthm}
\usepackage{amsfonts}
\usepackage{amscd}
\usepackage{graphicx}

\usepackage[colorlinks=true,
linkcolor=webgreen,
filecolor=webbrown,
citecolor=webgreen]{hyperref}

\definecolor{webgreen}{rgb}{0,.5,0}
\definecolor{webbrown}{rgb}{.6,0,0}

\usepackage{color}
\usepackage{fullpage}
\usepackage{float}

\DeclareMathOperator{\arctanh}{arctanh}
\DeclareMathOperator{\arccot}{arccot}

\begin{document}

\theoremstyle{plain}
\newtheorem{theorem}{Theorem}
\newtheorem{lemma}[theorem]{Lemma}
\newtheorem{proposition}[theorem]{Proposition}

\theoremstyle{definition}
\newtheorem{definition}[theorem]{Definition}
\newtheorem{corollary}[theorem]{Corollary}
\newtheorem{remark}[theorem]{Remark}
\newtheorem{example}[theorem]{Example}

\begin{center}
\vskip 1cm{\LARGE\bf
Evaluation of some alternating series\\ involving the binomial coefficients $C(4n,2n)$\\
\vskip .1in } \vskip 1cm \large
Kunle Adegoke\\
Department of Physics and Engineering Physics\\
Obafemi Awolowo University\\
Ile-Ife, Nigeria\\
\href{mailto:adegoke00@gmail.com}{\tt adegoke00@gmail.com} \\
\vskip .25 in
Robert Frontczak\\
Independent Researcher\\
Reutlingen, Germany\\
\href{mailto:robert.frontczak@web.de}{\tt robert.frontczak@web.de}\\
\vskip .25 in
Taras Goy\\
Faculty of Mathematics and Computer Science\\
Vasyl Stefanyk Precarpathian National University \\
Ivano-Frankivsk, Ukraine\\
\href{mailto:taras.goy@pnu.edu.ua}{\tt taras.goy@pnu.edu.ua} \\
\end{center}

\vskip .4 in

\begin{abstract}
We evaluate in closed form several alternating infinite series involving the binomial coefficients
$C(4n,2n)$ and $C(4n+2,2n+1)$ in the denominator. One of our results generalizes an identity that
was obtained experimentally by Sprugnoli in 2006.
\end{abstract}

\vskip .1 in

\noindent 2020 {\it Mathematics Subject Classification}: Primary 40A30; Secondary 11B37, 11B39.

\vskip .05 in

\noindent \emph{Keywords:} Series; binomial coefficient; Fibonacci numbers; Lucas numbers; Catalan numbers; odd harmonic numbers.

\section{Motivation}

By using a bisection formula and a Computer Algebra System, R.~Sprugnoli \cite{sprugnoli06} evaluated the constant
\begin{equation*}
\sum_{n = 0}^\infty {( - 1)^n \binom{4n}{2n}^{ - 1} }\approx .84660943.
\end{equation*}
According to Sprugnoli, his result is a rather complicated formula, containing complex numbers; which he was however able to reduce to the following expression, containing only real quantities:
\begin{align}\label{sprug}
\begin{split}
\sum_{n = 0}^\infty  &( - 1)^n \binom{4n}{2n}^{ - 1}\\
&\hspace{-.7cm}=\frac{16}{17} + \frac{2\sqrt{34}}{289}\!
\left(\frac{2 (\sqrt {17}  - 2)}{\sqrt {\sqrt {17}  - 1}}\arctan\! \Big( \frac{\sqrt 2 }{\sqrt {\sqrt {17}  - 1}} \Big) +  \frac{\sqrt {17}  + 2}{\sqrt {\sqrt {17}  + 1} }\ln\!\Big( {\frac{{\sqrt {\sqrt {17}  + 1}  - \sqrt 2 }}{{\sqrt {\sqrt {17}  + 1}  + \sqrt 2 }}}\Big)\! \right)\!.
\end{split}
\end{align}

In this paper, among other things, we will generalize Sprugnoli's identity \eqref{sprug}.
In particular, we will show that
\begin{equation*}
\sum_{n = 1}^\infty  \frac{( - 1)^{n - 1}}{n^{2}}\binom{4n}{2n}^{ - 1}x^{2n}
 = 8\Big(\arctanh^2\! {\sqrt y } - \arctan^2\!  {\sqrt y }\Big),\qquad |x|\le 4,
\end{equation*}
where
\begin{equation*}
y = y(x) = \frac{\sqrt{x^2 + 16} - 4}{x}.
\end{equation*}
For $|x|\le 4$, we also have
\begin{equation*}
\sum_{n = 1}^\infty \frac{( - 1)^{n - 1}}{n} \binom{4n}{2n}^{ - 1}x^{2n}
= 4\sqrt y\,\,\frac{y^2 - 1}{y^2 + 1}\left( \frac{{\arctan {\sqrt y }}}{{y + 1}} + \frac{{\arctanh {\sqrt y }}}{y - 1} \right)
\end{equation*}
and
\begin{equation}
\begin{split}\label{eq.gof0o22}
\sum_{n = 1}^\infty ( - 1)^{n - 1} \binom{4n}{2n}^{- 1} x^{2n} &= \frac{4y^2}{(y^2  + 1)^2}
+ \left( \frac{{(y - 1)^3 }}{(y^2 + 1)^2 } - \frac{8(y - 1)y^2}{(y^2 + 1)^3} \right)\sqrt{y}\arctan {\sqrt y } \\
&\quad+\left( \frac{{(y + 1)^3 }}{(y^2 + 1)^2} - \frac{{8(y + 1)y^2}}{(y^2 + 1)^3} \right)\sqrt{y}\arctanh {\sqrt y }.
\end{split}
\end{equation}

Sprugnoli's result can be obtained by setting $x=1$ in \eqref{eq.gof0o22} in view of $
\arctanh x = \frac{1}{2}\ln \!\big( {\frac{1 + x}{1 - x}} \big)$.

Furthermore, evaluating the above series for carefully selected values of $x$ will produce additional interesting identities.
For instance, since $y(3)=1/3$, we immediately obtain
\begin{align*}
\sum_{n = 1}^\infty ( - 1)^{n - 1} \frac{ 9^n}{n^2}\binom{4n}{2n}^{ - 1} &= - \frac{{2\pi ^2 }}{9}+ 2\ln ^2\! \big( 2 + \sqrt 3 \big),\\
\sum_{n = 1}^\infty ( - 1)^{n - 1} \frac{9^n}{n}\binom{4n}{2n}^{ - 1}  &= - \frac{{2\sqrt 3\,\pi }}{15} + \frac{{4\sqrt 3 }}{5}\ln \big( 2+\sqrt 3\big),
\end{align*}
\begin{align}
\sum_{n = 1}^\infty ( - 1)^{n - 1} 9^n \binom{4n}{2n}^{ - 1} &= \frac{9}{25} + \frac{{4 \sqrt 3\,\pi }}{375} + \frac{{22\sqrt 3 }}{125}\ln \big( 2+ \sqrt 3\big).\nonumber
\end{align}

In addition, since $y({1}/{\sqrt3})=(2-\sqrt3)^2$ and $\arctan(2-\sqrt3)={\pi}/{12}$, we have
\begin{align*}
\sum_{n=1}^\infty \frac{(-1)^{n-1}}{n^2 3^n} \binom{4n}{2n}^{- 1} &= - \frac{\pi^2}{18} + \frac{1}{2} \ln^2 3 , \\
\sum_{n=1}^\infty \frac{(-1)^{n-1}}{n 3^n} \binom{4n}{2n}^{- 1} &=  - \frac{\sqrt3}{21}\, \pi +\frac{2}{7} \ln3, \\
\sum_{n=1}^\infty \frac{(-1)^{n-1}}{3^n} \binom{4n}{2n}^{- 1} &= \frac{1}{49}  - \frac{10 \sqrt{3}}{1029}\pi+ \frac{27}{343} \ln3.
\end{align*}

We note that since $\binom{4n}{2n}=(2n+1)C_{2n}$, where $C_{2n}$ are Catalan numbers, our results could be stated equivalently in terms of the Catalan numbers. Similar series were studied recently by the authors \cite{AFG1,AFG2}, Chu and Esposito \cite{ChuEspo},
Bhandari \cite{bhandari}, Chen \cite{Chen}, and Stewart \cite{stewart}, among others.

We will also establish connections with the Fibonacci and Lucas numbers. As usual, the Fibonacci numbers $F_n$
and the Lucas numbers $L_n$ are defined, for \text{$n\in\mathbb Z$}, through the recurrence relations
$F_n = F_{n-1}+F_{n-2}$, $n\ge 2$, with initial values $F_0=0$, $F_1=1$ and $L_n = L_{n-1}+L_{n-2}$ with $L_0=2$, $L_1=1$.
For negative subscripts we have $F_{-n}=(-1)^{n-1}F_n$ and $L_{-n}=(-1)^n L_n$. They possess the Binet's formulas
\begin{equation*}
F_n = \frac{\alpha^n - \beta^n }{\alpha - \beta },\qquad L_n = \alpha^n + \beta^n,\qquad n\in\mathbb Z,
\end{equation*}
with $\alpha=(1+\sqrt{5})/2$ being the golden ratio and $\beta=-1/\alpha$.
For more information we refer to the  On-Line Encyclopedia of Integer Sequences \cite{OEIS} where the sequences
$(F_n)_{n\geq 0}$ and $(L_n)_{n\geq 0}$ are indexed as entries A0100045 and A000032, respectively.

\section{Main results}

\begin{theorem}\label{main_thm1}
For $|x|\leq 4$ and
\begin{equation}\label{eq.m4bsgj7}
y = y(x) = \frac{{\sqrt {x^2  + 16} - 4}}{x},
\end{equation}
the following series identities hold true:
\begin{align}\label{eq.c9vae0a}
&\sum_{n = 1}^\infty \frac{( - 1)^{n - 1}}{n^2} \binom{4n}{2n}^{-1 } x^{2n} = 8\left(\arctanh^2 {\sqrt y } - \arctan^2 {\sqrt y}\right)\,,\\
\label{eq.ccu9ufs}
&\sum_{n = 1}^\infty \frac{( - 1)^{n - 1}}{n} \binom{4n}{2n}^{ - 1} x^{2n} = \frac{4\sqrt y\,(y^2 - 1)}{y^2 + 1}\left( {\frac{\arctan {\sqrt y }}{y + 1} + \frac{\arctanh\sqrt y }{{y - 1}}} \right)\!,
\end{align}
and
\begin{equation}\label{eq.genSprug}
\begin{split}
\sum_{n = 1}^\infty ( - 1)^{n - 1} \binom{4n}{2n}^{ - 1} x^{2n} &= \frac{4y^2}{(y^2  + 1)^2} + \left( {\frac{{(y - 1)^3 }}{{(y^2  + 1)^2 }}
- \frac{{8(y - 1)y^2}}{(y^2  + 1)^3 }} \right)\sqrt{y}\arctan {\sqrt y } \\
&\quad+ \left( {\frac{{(y + 1)^3 }}{{(y^2  + 1)^2 }}} - \frac{8(y + 1)y^2}{(y^2  + 1)^3} \right)\sqrt{y}\arctanh{\sqrt y}.
\end{split}
\end{equation}
\end{theorem}
\begin{proof}
We start with the identity \cite[Identity (13)]{lehmer85}
\begin{equation}\label{eq.niu3t37}
\sum_{n = 1}^\infty \frac{2^{2n}}{n^{2}}  \binom{2n}n^{ - 1} x^{2n} =  2\arcsin^2 x,\quad |x|\le 2.
\end{equation}
Let $x$ be a non-zero real variable. It is not difficult to show that
\begin{equation}\label{eq.h3h5tiw}
\Re\Big( {\arcsin \big(x\sqrt i\,\big)} \Big) = \arctan \sqrt {\frac{{\sqrt {1 + x^4 } - 1}}{{x^2 }}}
\end{equation}
and
\begin{equation}\label{eq.lebm8kb}
 \Im\Big( {\arcsin \big(x\sqrt i\,\big)} \Big) = \arctanh \sqrt {\frac{{\sqrt {1 + x^4 } - 1}}{{x^2 }}} ,
\end{equation}
where $\Re(X)$ and $\Im(X)$ denote, respectively, the real and the imaginary part of $X$ and $i=\sqrt{-1}$ is the imaginary unit.
Writing $x\sqrt i$ for $x$ in \eqref{eq.niu3t37} and taking the real part using \eqref{eq.h3h5tiw} and \eqref{eq.lebm8kb}
and the fact that $\Re (f^2)=(\Re f)^2-(\Im f)^2$ for arbitrary $f$ gives after replacing $x^2$ with $x$
\begin{equation*}
\sum_{n = 1}^\infty \frac{(- 1)^{n - 1}4^{2n} }{n^{2}}   \binom{4n}{2n}^{ - 1} x^{2n} =
8 \arctanh^2 \sqrt{\frac{\sqrt{1+x^2}-1}{x}} - 8\arctan^2 \sqrt{\frac{\sqrt{1+x^2}-1}{x}}.
\end{equation*}
Identity \eqref{eq.c9vae0a} follows immediately after replacing $x$ by $x/4$ and simplifying.
Noting that
\begin{equation*}
x\frac{dy}{dx} = \frac{1 - y^2}{1 + y^2 }y,
\end{equation*}
and differentiating \eqref{eq.c9vae0a}, after some algebra, the other two results follow.
\end{proof}
\begin{remark}
	We emphasize that in the formulas of Theorem \ref{main_thm1}, the arctangent and inverse hyperbolic tangent functions
	can have a complex number as an argument.  For instance, as $y(4i)=	i$, from  \eqref{eq.c9vae0a} we have
	\begin{align*}
	\sum_{n=1}^\infty \frac{16^{n}}{n^2} \binom{4n}{2n}^{-1} &=  8\left(\arctan^2\sqrt{ i} + \arctan^2\sqrt{-i}\right) \\
	&= -2 \left(\ln^2 \Big(\frac{i}{\sqrt{2}-1}\Big) + \ln^2 \Big(\frac{i}{1-\sqrt{2}}\Big)\right)\\
	&= \pi^2 -4 \ln^2 \big(\sqrt{2}-1\big).
	\end{align*}

	More generally,
	\begin{align*}
	\sum_{n=1}^\infty  \binom{4n}{2n}^{-1} \frac{x^{2n}}{n^2} &=8\left( \arctan^2\sqrt{ y(ix)}-\arctanh^2\sqrt{ y(ix)}\right) \\
	&= 8\left(\arctan^2 \sqrt{ y(ix)} + \arctan^2 \big(i\sqrt{ y(ix)}\big)\right)\\
	&= 8\left(\arctan^2 \sqrt{ y(ix)} - \frac{1}{4}\ln^2 \Big(\frac{1+\sqrt{y(ix)}}{1-\sqrt{y(ix)}}\Big)\right)
	\end{align*}
	or
	\begin{align*}
	\sum_{n=1}^\infty  \binom{4n}{2n}^{-1} \frac{x^{2n}}{n^2} &= 8\left(\arctan^2 \sqrt{ Y_1(x)} - \arctanh^2 \sqrt{ Y_1(x)}\right)\\
	&= 8\left(\arctan^2 \sqrt{ Y_1(x)} - \frac{1}{4}\ln^2 \Big(\frac{1+\sqrt{Y_1(x)}}{1-\sqrt{Y_1(x)}}\Big)\right)\!,
	\end{align*}
	where
	\begin{equation*}
	Y_1(x)=\frac{4-\sqrt{16-x^2}}{x}\cdot i.
	\end{equation*}
\end{remark}
\begin{remark}
Inverting \eqref{eq.m4bsgj7} gives
\begin{equation*}
x(y) = \frac{8y}{1 - y^2};
\end{equation*}
so that \eqref{eq.c9vae0a} is equivalent to
\begin{equation}\label{eq.f6hxcke}
\sum_{n = 1}^\infty\frac{{( - 1)^{n - 1} 8^{2n - 1} }}{{n^2 }}\binom{4n}{2n}^{ - 1}\left( {\frac{y}{{1 - y^2 }}} \right)^{2n}
= \arctanh^2 \sqrt y -\arctan^2 \sqrt y
\end{equation}
which holds for all $y$ with $2|y|\le |1 - y^2|$, i.e., for all $y$ with $y\leq -1-\sqrt{2}$ or $y \geq 1+\sqrt{2}$
or $1-\sqrt{2}\leq y \leq \sqrt{2}-1$. Also, we have
\begin{equation}\label{eq.z1vtemc}
\sum_{n = 1}^\infty  {\frac{{( - 1)^{n - 1} 8^{2n - 1} }}{n}\binom{4n}{2n}^{ - 1} }\left( \frac{y}{1 - y^2} \right)^{2n}   = \frac{\sqrt y(1 - y^2)}{2(1 + y^2)} \left( \frac{\arctanh \sqrt y }{1 - y} - \frac{\arctan\sqrt y }{1 + y} \right),
\end{equation}
as well as
\begin{equation}\label{eq.rx3n6il}
\begin{split}
\sum_{n = 1}^\infty & ( - 1)^{n - 1} 8^{2n} \binom{4n}{2n}^{ - 1} \left( {\frac{y}{{1 - y^2 }}} \right)^{2n}  \\
&\qquad\qquad=(1 - y)(1 + 4y + y^2 )\frac{1 - y^2 }{{(1 + y^2 )^3 }} \sqrt{y} \arctanh \sqrt y\\
&\qquad\qquad\quad\, - (1 + y)(1 - 4y + y^2 )\frac{1 - y^2}{{(1 + y^2 )^3 }} \sqrt{y} \arctan \sqrt y+\frac{4y^2}{(1 + y^2 )^2}.
\end{split}
\end{equation}
\end{remark}

Setting $y=\tan(\theta/2)$ and using the fact that
\begin{equation*}
\sin \theta = \frac{{2\tan (\theta/2)}}{{1 + \tan ^2 (\theta/2)}},\quad\cos \theta = \frac{{1 - \tan ^2 (\theta/2)}}{{1 + \tan ^2 (\theta/2)}},\quad\tan \theta = \frac{{2\tan (\theta/2)}}{{1 - \tan ^2 (\theta/2)}},
\end{equation*}
we obtain the trigonometric versions of identities \eqref{eq.f6hxcke}--\eqref{eq.rx3n6il}.
\begin{proposition}\label{prop.nrbtka2}
If $|\tan\theta|\le 1$, then
\begin{align*}
\sum_{n = 1}^\infty\, & \frac{{( - 1)^{n - 1} 16^{n} }}{{n^2 }}\binom{4n}{2n}^{ - 1} \tan ^{2n} \theta   = 8\arctanh^2 \sqrt {\tan \frac{\theta}2} - 8\arctan^2 \sqrt {\tan \frac{\theta}2},\\
\sum_{n = 1}^\infty \, &\frac{( - 1)^{n - 1} 16^{n} }{n}\binom{4n}{2n}^{ - 1} \tan ^{2n} \theta\\
& = 2\sqrt {\tan \frac{\theta}2}\left( (\sin \theta  + \cos \theta  + 1)  \arctanh \sqrt {\tan \frac{\theta}2}+ \left( {\sin \theta  - \cos \theta  - 1} \right) \arctan \sqrt{\tan \frac{\theta}2}\right),
\end{align*}
and
\begin{equation*}
\begin{split}
\sum_{n = 1}^\infty & ( - 1)^{n - 1} 16^{n} \binom{4n}{2n}^{ - 1} \tan ^{2n} \theta\\
&= \frac12( {\sin 2\theta + \cos 2\theta + \sin \theta  + \cos \theta })\cos \theta \sqrt {\tan \frac{\theta}2} \arctanh \sqrt {\tan \frac{\theta}2} \\
&\quad  +\frac12 ( {\sin 2\theta - \cos 2\theta + \sin \theta  - \cos \theta } )\cos \theta \sqrt {\tan \frac{\theta}2} \arctan \sqrt {\tan \frac{\theta}2} + \sin ^2 \theta.
\end{split}
\end{equation*}
\end{proposition}
Evaluation at $\theta=\pi/4$ gives the next set of results.
\begin{proposition}
We have
\begin{align*}
\sum_{n = 1}^\infty & \frac{( - 1)^{n - 1} 16^{n} }{{n^2 }}\binom{4n}{2n}^{ - 1}  = 2\ln^2 \Big( \sqrt 2  + 1+\sqrt {2(\sqrt 2  + 1)}\Big)  - 8 \arctan^2 \sqrt {\sqrt 2  - 1},\\
\sum_{n = 1}^\infty&\frac{{( - 1)^{n - 1} 16^{n} }}{n}\binom{4n}{2n}^{ - 1} \\
& \qquad= \sqrt {\sqrt 2  + 1} \ln \Big(\sqrt 2  + 1+\sqrt {2(\sqrt 2  + 1)}\Big)- 2\sqrt {\sqrt 2  - 1}\, \arctan \sqrt {\sqrt 2  - 1},
\end{align*}
and
\begin{align*}
\sum_{n = 1}^\infty&  ( - 1)^{n - 1} 16^{n} \binom{4n}{2n}^{ - 1} \\
  &=\frac12 + \frac{\sqrt {2(\sqrt 2  + 1)}}{8} \ln \Big( \sqrt 2  + 1+ \sqrt {2(\sqrt 2  + 1)}\Big)+\, \frac{\sqrt {2(\sqrt 2  - 1)}}{4}\arctan \sqrt {\sqrt 2  - 1}.
\end{align*}
\end{proposition}
\begin{remark}
	Using $\arctan\frac{1}{x}=\arccot x$ for $x>0$, we can express each of the above identities in a simpler way. For instance,
	\begin{align*}
	\sum_{n=1}^\infty \frac{(-1)^{n-1} 16^{n}}{n^2} \binom{4n}{2n}^{-1} =2  \ln^2\big(\delta+\sqrt{2\delta}\big) - 8\arccot^2 \sqrt{\delta},
	\end{align*}
	where $\delta=\sqrt2+1$ is the silver ratio.
\end{remark}

Choosing $\theta=\pi/6$ in Proposition \ref{prop.nrbtka2} produces the following results.
\begin{proposition}
We have
\begin{align*}
&\sum_{n = 1}^\infty  {\frac{{( - 1)^{n - 1} }}{{n^2 }}\left( {\frac{{16}}{3}} \right)^n \binom{4n}{2n}^{ - 1} }  = 2\ln ^2 \big( {\sqrt 2  + \sqrt 3 } \big) - 8\arctan ^2 \Big( \frac{\sqrt 6  - \sqrt 2}{2} \Big),\\
&\sum_{n = 1}^\infty  {\frac{{( - 1)^{n - 1} }}{n}\left( {\frac{{16}}{3}} \right)^n \binom{4n}{2n}^{ - 1} }  = \frac{\sqrt 6}{2}\ln \big( {\sqrt 2  + \sqrt 3 } \big) - \sqrt 2\arctan \Big( \frac{\sqrt 6  - \sqrt 2}{2} \Big),
\end{align*}
and
\begin{equation*}
\sum_{n = 1}^\infty  {( - 1)^{n - 1} \left( {\frac{{16}}{3}} \right)^n \binom{4n}{2n}^{ - 1} }  = \frac{1}{4} + \frac{\sqrt 6}{8}  \,\ln \big( {\sqrt 2  + \sqrt 3 } \big).
\end{equation*}
\end{proposition}

Some results involving Fibonacci ( Lucas) numbers and the golden ratio are presented in the next proposition.
\begin{proposition}
	Let $r$ be a positive integer. If $r$ is even with $r\ne0$, then
	\begin{align}\label{eq.uuuu8t9}
	&\sum_{n = 1}^\infty \frac{{(- 1)^{n - 1}}}{n^2} \left( \frac{64}{5F_r^2} \right)^n \binom{4n}{2n}^{- 1}
	= 2\ln^2 \Big(\frac{\sqrt{\alpha^r}+1}{\sqrt{\alpha^r}-1} \Big) - 8 \arccot^2 \sqrt{\alpha^r},\\
	\begin{split}
	&\sum_{n = 1}^\infty \frac{(- 1)^{n - 1}}{n} \left( \frac{64}{5F_r^2} \right)^n \binom{4n}{2n}^{- 1}\\
	&\qquad\qquad\qquad\qquad=  \frac{2}{\sqrt{\alpha^r}L_r}\left((\alpha^r+1)\ln\Big(\frac{\sqrt{\alpha^r}+1}{\sqrt{\alpha^r}-1}\Big)-2(\alpha^r-1)\arccot\sqrt{\alpha^r}\right)\!,\label{last}
	\end{split}
	\end{align}
	and
	\begin{align}\label{eq.xc5fj25}
	\begin{split}
	&\sum_{n=1}^\infty (-1)^{n-1} \Big(\frac{64}{5F_r^2}\Big)^n \binom{4n}{2n}^{-1}\\
	&\qquad\qquad\qquad=\frac{4}{L^2_r} + \frac{5\sqrt{\alpha^r}F^2_r}{2L^3_r}\left(\frac{L_r+4}{\alpha^r+1}\ln\Big(\frac{\sqrt{\alpha^r}+1}{\sqrt{\alpha^r}-1}\Big)-\frac{2(L_r-4)}{\alpha^r-1}\arccot\sqrt{\alpha^r}\right)\!.
	\end{split}
	\end{align}

	If $r$ is odd, then
	\begin{align*}
	&\sum_{n=1}^\infty \frac{(-1)^{n-1}}{n^2} \Big(\frac{64}{L_r^2}\Big)^n \binom{4n}{2n}^{-1}
	= 2\ln^2\Big(\frac{\sqrt{\alpha^r}+1}{\sqrt{\alpha^r}-1}\Big) - 8 \arccot^2\sqrt{\alpha^r}, \\
	&\sum_{n=1}^\infty \frac{(-1)^{n-1}}{n} \Big(\frac{64}{L_r^2}\Big)^n \binom{4n}{2n}^{-1} = \frac{2\sqrt{\alpha^r}L_r}{\sqrt5F_r}\left(\frac{1}{\alpha^r-1}\ln\Big(\frac{\sqrt{\alpha^r}+1}{\sqrt{\alpha^r}-1}\Big)-\frac{2}{\alpha^r+1} \arccot\sqrt{\alpha^r}\right)\!,\\
	&\sum_{n=1}^\infty (-1)^{n-1}\Big(\frac{64}{L_r^2}\Big)^n \binom{4n}{2n}^{-1}=\frac{4}{5F_r^2}\\
	&\quad+ \frac{\sqrt{5}L_r}{50\sqrt{\alpha^r}F^3_r}\left((\alpha^r-1)(\sqrt5F_r+4)\ln\Big(\frac{\sqrt{\alpha^r}+1}{\sqrt{\alpha^r}-1}\Big)-2(\alpha^r+1)(\sqrt5F_r-4) \arccot\sqrt{\alpha^r}\right)\!.
	\end{align*}
\end{proposition}
\begin{proof}
	Set $y=\alpha^r$ in \eqref{eq.f6hxcke}--\eqref{eq.rx3n6il} to obtain \eqref{eq.uuuu8t9}--\eqref{eq.xc5fj25} and $y=-\alpha^r$ to obtain
	the remaining three identities.
\end{proof}

\section{Alternating series involving $C(4n+2,2n+1)$}

The counterparts of \eqref{eq.c9vae0a}--\eqref{eq.genSprug} involving $\binom{4n+2}{2n+1}^{-1}$ are stated in the next theorem.
\begin{theorem}\label{main_thm2}
For $ |x|\leq 4$, the following series identity holds true
\begin{equation}\label{eq.main_bin21}
\sum_{n = 0}^\infty  \frac{(- 1)^n}{(2n+1)^2}  \binom{4n+2}{2n+1}^{- 1}x^{2n+1} =
4 \arctanh \sqrt y \,\arctan \sqrt y,
\end{equation}
where $y=y(x)$ is defined in \eqref{eq.m4bsgj7}. We also have
\begin{equation}\label{eq.main_bin22}
\sum_{n = 0}^\infty\frac{(- 1)^{n}}{2n+1}  \binom{4n+2}{2n+1}^{-1}x^{2n+1} =
 \frac{2 \sqrt{y}(1-y^2)}{1+y^2}\left ( \frac{\arctanh {\sqrt y}}{1+y} + \frac{\arctan  {\sqrt y} }{1-y}\right )
\end{equation}
and
\begin{align}\label{eq.main_bin23}
\sum_{n = 0}^\infty (- 1)^{n}  &\binom{4n+2}{2n+1}^{- 1} x^{2n+1} \nonumber \\
&\qquad= \frac{2y(1-y^2)}{(1+y^2)^2} +  \frac{(1+y)(1-y)^2 (y^2+4y+1)}{(1+y^2)^3}\sqrt{y} \arctan\sqrt{y} \nonumber \\
&\qquad\quad +  \frac{(1-y)(1+y)^2 (y^2-4y+1)}{(1+y^2)^3}\sqrt{y} \arctanh\sqrt{y}.
\end{align}
\end{theorem}
\begin{proof}
Use exactly the same ideas as in Theorem \ref{main_thm1}. Write $x\sqrt i$ for $x$ in \eqref{eq.niu3t37}
and take the imaginary part using \eqref{eq.h3h5tiw} and \eqref{eq.lebm8kb} and the fact that $\Im (f^2)=2\Re f\Im f$ for arbitrary $f$.
This yields \eqref{eq.main_bin21}. The other two identities follow  from \eqref{eq.main_bin21} by differentiation.
\end{proof}

As particular examples we have the next evaluations:
\begin{align*}
&\sum_{n = 0}^\infty  \frac{(- 9)^{n}}{(2n+1)^2}  \binom {4n+2}{2n+1}^{- 1} = \frac{\pi}{9} \ln\big(2+\sqrt{3}\big),\\
&\sum_{n = 0}^\infty  \frac{(-\frac 13)^{n}}{(2n+1)^2} \binom {4n+2}{2n+1}^{- 1}
= \frac{\pi}{4\sqrt{3}} \ln3,\\
&\sum_{n = 0}^\infty  \frac{(- 9)^{n}}{2n+1} \binom {4n+2}{2n+1}^{- 1}
= \frac{2\pi}{15\sqrt{3}} + \frac{1}{5\sqrt{3}} \ln\big(2+\sqrt{3}\big),\\
&\sum_{n = 0}^\infty \frac{(-\frac 13)^{n}}{2n+1}  \binom{4n+2}{2n+1}^{- 1}
= \frac{\pi}{7\sqrt{3}} + \frac{3}{14} \ln3,\\
&\sum_{n = 0}^\infty (- 9)^{n} \binom {4n+2}{2n+1}^{- 1}
= \frac{4}{25} + \frac{22\pi}{375\sqrt{3}} - \frac{4}{125\sqrt{3}} \ln\big(2+\sqrt{3}\big),
\end{align*}
and
\begin{equation*}
\sum_{n = 0}^\infty \left(-\frac{ 1}{3}\right)^n   \binom{4n+2}{2n+1}^{- 1}
=  \frac{12}{49} + \frac{9\sqrt3\,\pi}{343} + \frac{30\ln3}{343}.
\end{equation*}

Setting $y=\tan(\theta/2)$, we obtain the trigonometric versions of identities \eqref{eq.main_bin21}--\eqref{eq.main_bin23}.
\begin{proposition}\label{prop.onqzlvl}
If $|\tan\theta|\le 1$, then
\begin{align*}
\sum_{n = 0}^\infty & {\frac{{( - 16)^n }}{{(2n + 1)^2 }}\binom{4n + 2}{2n + 1}^{ - 1} \tan ^{2n + 1} \theta }  =  \arctanh \sqrt {\tan\frac{\theta}{2}}\,\arctan \sqrt {\tan\frac{\theta}{2}},\\
\sum_{n = 0}^\infty &\frac{( - 16)^n }{2n + 1}\binom{4n + 2}{2n + 1}^{ - 1} \tan ^{2n + 1} \theta \\
 &= \frac14\sqrt {\tan\frac{\theta}{2}} \left((1-\sin \theta  + \cos \theta)\arctanh \sqrt {\tan\frac{\theta}{2}} + (1+ \sin \theta  + \cos \theta)\arctan \sqrt {\tan\frac{\theta}{2}}\right)\!,
\end{align*}
and
\begin{equation*}
\begin{split}
\sum_{n = 0}^\infty   ( - 16)^n &  \binom{4n + 2}{2n + 1}^{ - 1} \tan ^{2n + 1} \theta \\
&=   \frac{\sin 2 \theta}{8}- \frac{\cos \theta}{8}\big( {\sin 2\theta - \cos 2\theta + \sin \theta  - \cos \theta } \big) \sqrt {\tan\frac{\theta}{2}} \arctanh \sqrt {\tan\frac{\theta}{2}} \\
&\quad+ \frac{\cos \theta}{8}\big( {\sin 2\theta + \cos 2\theta + \sin \theta  + \cos \theta } \big) \sqrt {\tan\frac{\theta}{2}} \arctan \sqrt {\tan\frac{\theta}{2}}.
\end{split}
\end{equation*}
\end{proposition}
Evaluation in Proposition \ref{prop.onqzlvl} at $\theta=\pi/4$ and $\theta=\pi/6$ gives the next set of results.
\begin{corollary} We have
\begin{align*}
&\sum_{n = 0}^\infty  {\frac{{( - 16)^n}}{{(2n + 1)^2 }}\binom{4n + 2}{2n + 1}^{ - 1} }  = \frac12\ln\Big( \sqrt 2  + 1 + \sqrt {2(\sqrt 2  + 1)}  \Big)\arctan\sqrt {\sqrt 2  - 1}  ,\\
&\sum_{n = 0}^\infty  {\frac{{( - 16)^n }}{2n + 1}\binom{4n + 2}{2n + 1}^{ - 1} } = \frac{\sqrt {\sqrt 2  - 1}}{4} \arctanh \sqrt {\sqrt 2  - 1} + \frac{\sqrt {\sqrt 2  + 1}}{4} \arctan \sqrt {\sqrt 2  - 1},\\
&\sum_{n = 0}^\infty  ( - 16)^n \binom{4n + 2}{2n + 1}^{ - 1} \\
 &\qquad\qquad=\frac18  - \frac{\sqrt {2(\sqrt 2  - 1)}}{16} \arctanh \sqrt {\sqrt 2  - 1}+ \frac{\sqrt {2(\sqrt 2  + 1)}}{16} \arctan \sqrt {\sqrt 2  - 1},
\end{align*}
and
\begin{align*}
&\sum_{n = 0}^\infty  \frac{( -\frac{16}{3})^n }{(2n + 1)^2 }\binom{4n + 2}{2n + 1}^{ - 1}   = \sqrt3\arctanh\sqrt{2-\sqrt3}\,\arctan\sqrt{2-\sqrt3},\\
&\sum_{n = 0}^\infty  \frac{( - \frac{16}{3})^n }{2n + 1}\binom{4n + 2}{2n + 1}^{ - 1}   = \frac{\sqrt6}{8}\Big(\arctanh\sqrt{2-\sqrt3}+\sqrt3\,\arctan\sqrt{2-\sqrt3}\Big),\\
&\sum_{n = 0}^\infty  \left(-\frac{16}{3}\right)^{n}\binom{4n + 2}{2n + 1}^{ - 1}   = \frac{3}{16}\Big(1 + \sqrt2\arctan\sqrt{2-\sqrt3}\Big).
\end{align*}
\end{corollary}

In the next theorem we derive some binomial series involving Fibonacci and Lucas numbers.
\begin{theorem} \label{Th12}
Let the constants $A$ and $B$ be defined by
\begin{equation*}
A=\sqrt{\frac{\sqrt{70+2\sqrt5}-8}{1+\sqrt5}}\qquad\mbox{and} \qquad B=\sqrt{\frac{\sqrt{70-2\sqrt5}-8}{1-\sqrt5}}.
\end{equation*}
Then, for any integer $s$,  we have
\begin{align}
\label{Th12_1}
\sum_{n = 0}^\infty  \frac{( - 1)^n L_{2n+s}}{(2n+1)^2}\binom{4n + 2}{2n + 1}^{ - 1}= 4\left(\alpha^{s-1}\arctan A\arctanh A + \beta^{s-1}\arctan B\arctanh B\right)\!,
\end{align}
\begin{equation}
\label{Th12_2}
\begin{split}
&\sum_{n = 0}^\infty  \frac{( - 1)^n F_{2n+s}}{(2n+1)^2}\binom{4n + 2}{2n + 1}^{ - 1}\\
&\qquad\qquad\qquad= \frac{4}{\sqrt5}\left(\alpha^{s-1}\arctan A\arctanh A- 
\beta^{s-1}\arctan B\arctanh B\right)\!,
\end{split}
\end{equation}
\begin{align}
&\sum_{n = 0}^\infty  \frac{( - 1)^n L_{2n+s}}{2n+1}\binom{4n + 2}{2n + 1}^{ - 1}\nonumber \\
&\quad= 2A\alpha^{s-1}\frac{1-A^4}{1+A^4}\left(\frac{\arctanh A}{1+A^2}+ \frac{\arctan A}{1-A^2}\right)
+ 2B\beta^{s-1}\frac{1-B^4}{1+B^4}\left(\frac{\arctanh B}{1+B^2}+ \frac{\arctan B}{1-B^2}\right)\!,\nonumber\\
&\sum_{n = 0}^\infty  \frac{( - 1)^n F_{2n+s}}{2n+1}\binom{4n + 2}{2n + 1}^{ - 1}\nonumber \\
&\quad= \frac{2A\alpha^{s-1}}{\sqrt5}\frac{1-A^4}{1+A^4}\left(\frac{\arctanh A}{1+A^2}+ \frac{\arctan A}{1-A^2}\right)
-\frac{2B\beta^{s-1}}{\sqrt5}\frac{1-B^4}{1+B^4}\left(\frac{\arctanh B}{1+B^2}+ \frac{\arctan B}{1-B^2}\right)\!,\nonumber
\end{align}
and
\begin{align*}
\sum_{n = 0}^\infty ( - 1)^n \binom{4n + 2}{2n + 1}^{ - 1}L_{2n+s}&= \frac{2A^2(1-A^4)}{(1+A^4)^2}\alpha^{s-1} + \frac{2B^2(1-B^4)}{(1+B^4)^2}\beta^{s-1}\\
&\quad+ A\alpha^{s-1}\frac{(1+A^2)(1-A^2)^2(A^4+4A^2+1)}{(1+A^4)^3}\arctan A \\
&\quad+ B\beta^{s-1}\frac{(1+B^2)(1-B^2)^2(B^4+4B^2+1)}{(1+B^4)^3}\arctan B\\
&\quad+ A\alpha^{s-1}\frac{(1-A^2)(1+A^2)^2(A^4-4A^2+1)}{(1+A^4)^3}\arctanh A \\
&\quad+ B\beta^{s-1}\frac{(1-B^2)(1+B^2)^2(B^4-4B^2+1)}{(1+B^4)^3}\arctanh B,\\
\sqrt5 \sum_{n = 0}^\infty  ( - 1)^n \binom{4n + 2}{2n + 1}^{ - 1}F_{2n+s}&=\frac{2A^2(1-A^4)}{(1+A^4)^2}\alpha^{s-1}  - \frac{2B^2(1-B^4)}{(1+B^4)^2}\beta^{s-1}\\
&\quad+ A\alpha^{s-1}\frac{(1+A^2)(1-A^2)^2(A^4+4A^2+1)}{(1+A^4)^3}\arctan A \\
&\quad- B\beta^{s-1}\frac{(1+B^2)(1-B^2)^2(B^4+4B^2+1)}{(1+B^4)^3}\arctan B\\
&\quad+ A\alpha^{s-1}\frac{(1-A^2)(1+A^2)^2(A^4-4A^2+1)}{(1+A^4)^3}\arctanh A \\
&\quad- B\beta^{s-1}\frac{(1-B^2)(1+B^2)^2(B^4-4B^2+1)}{(1+B^4)^3}\arctanh B.
\end{align*}
\end{theorem}
\begin{proof}
Use $x=\alpha$ and $x=\beta$ in \eqref{eq.main_bin21}, in view of $y(\alpha)=A^2$ and $y(\beta)=B^2$, produce
\begin{equation}\label{1}
\sum_{n = 0}^\infty  \frac{(- 1)^n}{(2n+1)^2}  \binom{4n+2}{2n+1}^{- 1}\alpha^{2n+s} =
4\alpha^{s-1} \arctanh A\,\arctan A,
\end{equation}
and
\begin{equation}\label{2}
\sum_{n = 0}^\infty  \frac{(- 1)^n}{(2n+1)^2}  \binom{4n+2}{2n+1}^{- 1}\beta^{2n+s} =
4\beta^{s-1} \arctanh B\,\arctan B.
\end{equation}
Addition of \eqref{1} and \eqref{2} yields \eqref{Th12_1} while their difference gives \eqref{Th12_2}.
The other formulas are proven in the same manner; hence, the proofs are omitted.
\end{proof}
\begin{example}
For $s\in\{0,1,2\}$, from first and second formulas of Theorem \ref{Th12} we have the following series list:
\begin{align*}
&\sum_{n = 0}^\infty  \frac{( - 1)^{n+1}}{(2n+1)^2} \binom{4n + 2}{2n + 1}^{ - 1} L_{2n}= 4\big(\beta\arctan A\arctanh A +\alpha\arctan B\arctanh B\big),\\
&\sum_{n = 0}^\infty  \frac{( - 1)^{n+1}}{(2n+1)^2} \binom{4n + 2}{2n + 1}^{ - 1} F_{2n} =\frac{4}{\sqrt5}\big(\beta\arctan A\arctanh A -\alpha\arctan B\arctanh B\big),\\
&\sum_{n = 0}^\infty  \frac{( - 1)^{n}}{(2n+1)^2} \binom{4n + 2}{2n + 1}^{ - 1} L_{2n+1}= 4\big(\arctan A\arctanh A +\arctan B\arctanh B\big),\\
&\sum_{n = 0}^\infty  \frac{( - 1)^{n}}{(2n+1)^2} \binom{4n + 2}{2n + 1}^{ - 1} F_{2n+1} =\frac{4}{\sqrt5}\big(\arctan A\arctanh A -\arctan B\arctanh B\big),
\end{align*}
and
\begin{align*}
&\sum_{n = 0}^\infty\frac{( - 1)^{n}}{(2n+1)^2} \binom{4n + 2}{2n + 1}^{ - 1} L_{2n+2}= 4\big(\alpha\arctan A\arctanh A +\beta\arctan B\arctanh B\big),\\
&\sum_{n = 0}^\infty\frac{( - 1)^n}{(2n+1)^2} \binom{4n + 2}{2n + 1}^{ - 1} F_{2n+2} =\frac{4}{\sqrt5}\big(\alpha\arctan A\arctanh A -\beta\arctan B\arctanh B\big).
\end{align*}
\end{example}

\section{Other results: series involving Catalan numbers and second order odd harmonic numbers}

In this section we will work with the following series (\cite[page 262/263]{berndt},\cite{borwein})
\begin{equation}\label{cube_arcsin}
\sum_{n = 0}^\infty \frac{{2n + 1}}{{2n + 3}} \frac{C_n O_n^{(2)}}{4^{n} } x^{2n + 3} = \frac{1}{3}\arcsin^3 x, \quad |x|\leq 1,
\end{equation}
where $C_k = \frac{1}{{k + 1}}\binom{2k}{k}$ are the Catalan numbers and
\begin{equation*}
O_k^{(2)} = \sum_{j = 0}^k \frac{1}{(2j + 1)^2},\quad k\ge 0,
\end{equation*}
are the second order odd harmonic numbers.

\begin{theorem}\label{main_thm3}
	For all $x\in\mathbb{C}$ with $|x|\le 1/4$, the following series identity holds true
	%
	\begin{align*}
	\sum_{n = 0}^\infty &(- 1)^{n+1} \frac{4n+1}{4n+3} O_{2n}^{(2)} C_{2n} x^{2n}\\
	& = \frac{ \arctan {\sqrt z } + \arctanh {\sqrt z } }{24\sqrt{2x^3}}
	\Big ( \arctan^2 {\sqrt z } -4 \arctan {\sqrt z } \arctanh {\sqrt z }+\arctanh^2 {\sqrt z }\Big ),
	\end{align*}
	where
	\begin{equation*}
	z = z(x) = \frac{\sqrt{1+16x^2} - 1}{4x}.
	\end{equation*}
	In addition, we have for all $x\in\mathbb{C}$ with $|x|\le 1/4$
	\begin{align*}
	\sum_{n = 0}^\infty &(- 1)^{n+1}  \frac{4n+3}{4n+5} O_{2n}^{(2)} C_{2n} x^{2n}\\
	& = \frac{ \arctan {\sqrt z } - \arctanh {\sqrt z } }{24\sqrt{2x^5}}
	\Big ( \arctan^2 {\sqrt z } +4 \arctan {\sqrt z } \arctanh {\sqrt z }+\arctanh^2 {\sqrt z }\Big ).
	\end{align*}
\end{theorem}
\begin{proof}
	Starting with \eqref{cube_arcsin} and noting that $(\sqrt{i})^{-3}=-\frac1{\sqrt{2}}(1+i)$ we can write
	\begin{align*}
	\sum_{n = 0}^\infty (- 1)^{n} \frac{4n+1}{4n+3} \frac{O_{2n}^{(2)} C_{2n}}{2^{4n+1}} x^{4n}
	+ i \sum_{n = 0}^\infty (- 1)^{n} \frac{4n+3}{4n+5} \frac{O_{2n+1}^{(2)} C_{2n+1} }{2^{4n+3}} x^{4n+2} = g(x) f^3(x)
	\end{align*}
	with
	\begin{equation*}
	g(x) = - \frac{1+i}{6 \sqrt{2}}{x^{-3}} \qquad\mbox{and}\qquad f(x) = \arcsin(\sqrt{i} x).
	\end{equation*}
	Next, we take the real and imaginary parts for the product on the right-hand side using
	\begin{equation*}
	\Re f^3 = (\Re f)^3 - 3\Re f(\Im f)^2
	\end{equation*}
	and
	\begin{equation*}
	\Im f^3 = 3(\Re f)^2 \Im f - (\Im f)^3.
	\end{equation*}
	This yields
	\begin{equation*}
	\sum_{n = 0}^\infty (- 1)^{n}   \frac{4n+1}{4n+3} \frac{O_{2n}^{(2)} C_{2n}}{16^{n}} \,x^{4n}
	= - \frac{\sqrt{2}}{6  x^3}\left ( (\Re f)^3 + (\Im f)^3 - 3\, \Re f\,\Im f(\Re f + \Im f)\right )
	\end{equation*}
	and
	\begin{equation*}
	\sum_{n = 0}^\infty (- 1)^{n}  \frac{4n+3}{4n+5} \frac{O_{2n+1}^{(2)} C_{2n+1}}{16^{n}} \, x^{4n}
	= - \frac{2\sqrt{2}}{3  x^5}\left ( (\Re f)^3 - (\Im f)^3 + 3\, \Re f\,\Im f(\Re f - \Im f)\right)\!.
	\end{equation*}
	Make the transformation $x \mapsto 2\sqrt{x}$ and simplify while using \eqref{eq.h3h5tiw} and \eqref{eq.lebm8kb}.
\end{proof}

As particular examples of Theorem \ref{main_thm3} we offer
\begin{align*}
&\sum_{n = 0}^\infty (- 1)^{n+1}  \frac{4n+1}{4n+3} \frac{O_{2n}^{(2)} C_{2n}}{16^n}\\
&\qquad = \frac{\sqrt{2}}{6}
\big(\arctan\lambda  + \arctanh\lambda\big)
\big(\arctan^2\lambda-4\arctan\lambda\arctanh\lambda+\arctanh^2\lambda\big),\\
&\sum_{n = 0}^\infty (- 1)^{n+1}  \frac{4n+3}{4n+5} \frac{O_{2n+1}^{(2)} C_{2n+1}}{16^n}\\
&\qquad = \frac{2\sqrt{2}}{3}
\big(\arctan\lambda-\arctanh\lambda\big)\big(\arctan^2\lambda+4\arctan\lambda\arctanh\lambda+\arctanh^2\lambda\big),
\end{align*}
where $\lambda=\sqrt{\sqrt{2}-1}$, and
\begin{align*}
&\sum_{n = 0}^\infty (- 1)^{n+1}  \frac{4n+1}{4n+3} \frac{O_{2n}^{(2)} C_{2n}}{64^n} \\
&\qquad= \frac{2}{3}
\big(\arctan\omega  + \arctanh\omega\big)
\big(\arctan^2\omega-4\arctan\omega\arctanh\omega+\arctanh^2\omega\big),\\
&\sum_{n = 0}^\infty (- 1)^{n+1}  \frac{4n+3}{4n+5} \frac{O_{2n+1}^{(2)} C_{2n+1}}{64^n}\\
&\quad = \frac{16}{3}
\big(\arctan\omega-\arctanh\omega\big)\big(\arctan^2\omega+4\arctan\omega\arctanh\omega+\arctanh^2\omega\big),
\end{align*}
where $\omega=\sqrt{\sqrt5-2}=\sqrt{-\beta^3}$.

\section{Conclusion}

This paper was motivated by a challenging and mysterious identity stated without a formal proof by Sprugnoli in 2006.
We have shown that his identity is a special case of a much lager class of identities involving $\binom {4n}{2n}$
in the denominators. We have also covered similar series (with and without Fibonacci numbers) and derived various interesting identities.
In our next paper on the topic we will study some related series using the same approach.


\begin{thebibliography}{99}

\bibitem{AFG1}
K.~Adegoke, R.~Frontczak and T.~Goy, Fibonacci--Catalan series, \emph{Integers}
{\bf 22} (2022), \#A110.

\bibitem{AFG2}
K.~Adegoke, R.~Frontczak and T.~Goy, On a family of infinite series with reciprocal Catalan numbers,
\emph{Axioms} {\bf 11} (2022), 165.

\bibitem{berndt}
B.~C.~Berndt, \emph{Ramanujan's Notebooks, Part I, With a foreword by S. Chandrasekhar}, Springer-Verlag, New York, 1985.

\bibitem{bhandari}
N.~Bhandari, Infinite series associated with the ratio and product of central binomial coefficients,
\emph{J. Integer Seq.} {\bf 25} (2022), Article 22.6.5.

\bibitem{borwein}
J.~M.~Borwein and M.~Chamberland, Integer powers of Arcsin,
\emph{Int. J. Math. Math.  Sci.} {\bf 2007} (2007), Article ID 19381.

\bibitem{Chen}
H.~Chen, Interesting Ramanujan-like series associated with powers of central binomial coefficients,
\emph{J. Integer Seq.} {\bf 25} (2022), Article 22.1.8.

\bibitem{ChuEspo}
W.~Chu and F.~L.~Esposito, Evaluation of Ap\'{e}ry-like series through multisection method, \emph{J. Class. Anal.}
{\bf 12} (2018), 55--81.

\bibitem{lehmer85}
D.~H.~Lehmer, Interesting series involving the central binomial coefficient, \emph{Amer. Math. Monthly} {\bf 92} (1985), 449--457.

\bibitem{OEIS}
N.~J.~A. Sloane (ed.), \textit{The On-Line Encyclopedia of Integer Sequences}. Published electronically at \url{https://oeis.org}, 2023.

\bibitem{sprugnoli06}
R.~Sprugnoli, Sums of reciprocals of the central binomial coefficient, \emph{Integers} {\bf 6} (2006), \#A27.

\bibitem{stewart}
S. M. Stewart, A simple proof and some applications of an integral representation for the Catalan numbers,
\emph{Appl. Math. E-Notes} {\bf 22} (2022), 637--648.


\end{thebibliography}
\end{document}